\documentclass[12pt]{amsart}
\usepackage{geometry, amsmath, amsthm, amssymb} 
\newtheorem{thm}{Theorem}[section]
\newtheorem{cor}[thm]{Corollary}
\newtheorem{lem}[thm]{Lemma}
\newtheorem{prop}{Proposition}

\usepackage{verbatim}
\title{Codazzi Tensors with Two Eigenvalue Functions}

\author{Gabe Merton}

\begin{document}
\maketitle

\begin{abstract}
This paper addresses a gap in the classifcation of Codazzi tensors  with exactly two eigenfunctions on a Riemannian manifold of dimension three or higher.   Derdzinski proved that if  the trace of such a tensor is constant and the dimension of one of the the eigenspaces is $n-1$,  then the metric is a warped product where the base is an open interval- a conclusion we will show to be true under a milder trace condition.  Furthermore, we construct examples of Codazzi tensors having two eigenvalue functions, one of which has  eigenspace dimension $n-1$, where the metric is not a warped product with interval base, refuting a remark in \cite{Besse} that the warped product conclusion holds without any restriction on the trace.
\end{abstract}

\section{Introduction}

A symmetric $(0,2)$ tensor $A$ is Codazzi if it satisfies the symmetry property$$(\nabla_XA)(Y,Z) = (\nabla_Y A)(X,Z)$$ for any vector fields $X$, $Y$ and $Z$.   Alternatively, a $(1,1)$ tensor $A$ is Codazzi if it is self-adoint and $$\left(\nabla_X A \right)Y = \left(\nabla_Y A \right)X$$  Throughout the paper, $V_\lambda$ denotes the eigendistribution corresponding to the eigenvalue function $\lambda$ of the tensor $A$.  That is, we say a vector field $Y$ is in $V_\lambda$ if $AY = \lambda Y$.  While it's conceptually more appropriate to view Codazzi tensors as $(1,1)$, computations are often easier when the tensor is viewed as $(0,2)$. 
\vspace{1pc}

There are many well-known examples of Codazzi tensors: any constant scalar multiple of the metric, and more generally any parallel self-adjoint $(1,1)$ tensor.  One may also ask what it means if certain well known self-adjoint $(1,1)$ tensors are Codazzi.  For example, the second fundamental form of a hypersurface embedded in a space of constant sectional curvature is Codazzi.  When the larger space is $\mathbb{R}^3$, this is the content of the famous Codazzi-Mainardi equation.  It's a standard exercise in Riemannian geometry to show that $Ric$ is Codazzi if and only if the divergence of the full curvature tensor vanishes, i.e. the curvature is \emph{harmonic}.  This is the case, for example, on Einstein manifolds.\\

There are other more subtle relationships between the behavior of Codazzi tensors and  the topology and geometry of the manifold.  Berger-Ebin proved in \cite{BergerEbin} that a constant trace Codazzi tensor on a compact manifold with non-negative sectional curvature must be parallel.  As a nice low-dimension result, Bourguignon showed in \cite{Bourg} that a compact orientable four-manifold admitting a non-trivial Codazzi tensor with constant trace must have signature zero.   Derdzinski-Shen proved in \cite{DerdShen} that if a Codazzi tensor on $M^n$ has $n$ distinct eigenvalues at all points of $M$, then all the Pontryagin classes of $M$  are zero.  

Another result relating the geometry of the manifold  to a Codazzi tensor's spectrum is the departure point of this paper.

\begin{thm}  (Derdzinski) \label{thm:der}
  Suppose $A$ is a Codazzi tensor on $M^n$, $n \ge 3$ having exactly two distinct eigenvalue functions $\lambda, \mu$ in a neighborhood of $p$ with dim $V_\mu \le$ dim $V_\lambda$.  Then, there exists a neighborhood of $p$ such that
  \begin{itemize}
  	\item[i.]  $M$ is a Riemannian product if and only if dim $V_\mu \ge 2$ or $A$ is parallel in a neighborhood of $p$.
	\item[ii.]  $M$ is a warped product with interval base and non-trivial warping function if and only if dim $V_\mu = 1$, $A$ has constant trace, and $A$ is not parallel.  In this case, $M = I \times_{F} N$ where $N$ and $I$ are the integral submanifolds of $V_\lambda$ and $V_\mu$, respectively.
  \end{itemize}
\end{thm}
   It's not immediately obvious that such integral manifolds exists, however, a well known property of Codazzi tensors that we'll review in Section \ref{sec:background} is the fact that their eigendistributions are integrable.\\
  
	Theorem \ref{thm:der} was first proved by Derdzinski in \cite{Derdz} and is reported in Besse's \emph{Einstein Manifolds}, \cite{Besse}.  Besse precedes the proof with the statement "A similar argument works without the hypothesis [that trace $A$ is constant]."  This statement can be interpreted in two ways.  First some terminology.   If $A$ is a Codazzi tensor with exactly two distinct eigenfunctions $\mu$ and $\lambda$ and if $M = M_1 \times_F M_2$ is a warped product, we will say that the warped product and eigenspace structures are \emph{consistent} if $M_1$ and $M_2$ are integral submanifolds of the eigendistributions $V_\mu$ and $V_\lambda$.  In Theorem \ref{thm:der}, the structures are consistent.   Besse's remark could be saying that without the constant trace assumption,
	\begin{itemize}
		\item[i.]  $M$ is either a product or a warped product with interval base; or\\ 
		\item[ii.]  $M$ is either a product or a warped product with interval base and the warping variable is the coordinate of the interval.
		\end{itemize}



	In Section \ref{sec:counterexamples} , we prove the following propositions showing neither statement is true. 

\begin{prop}
\label{prop:notwarped}
There exists a compact Riemannian manifold $(M,g)$ and a Codazzi tensor $A$ with exactly two distinct eigenfunctions $\mu$ and $\lambda$ with dim $V_\mu = 1$ such that $M$ is neither a product nor a warped product.
\end{prop}

\begin{prop}\label{prop:inconsistent}
	There exists a Riemannian manifold $(M,g)$ and a Codazzi tensor $A$ with exactly two distinct eigenfunctions $\mu$ and $\lambda$ with dim $V_\mu = 1$ such that $M$ is a warped product with interval base but the warped product structure is inconsistent with the eigenspace structure.
\end{prop}

However, the conclusion of Theorem \ref{thm:der} does hold under weaker conditions than the trace being constant.  In Section \ref{sec:mainthm} we prove,\\

\begin{thm}\label{thm:mainthm}  Let $A$ be a Codazzi tensor on a manifold $M^n$, $n\ge 3$.  Suppose there exists a neighborhood of a point $p$ where $A$ has two distinct eigenfunctions $\mu$ and $\lambda$, dim $V_\mu = 1$, and $\lambda$ is not constant. Assume that any one of the following conditions hold:
\begin{itemize}
	\item[(1)]  $D_Y\left( tr(A)\right) = 0$ for all $Y \in V_\lambda$.
	\item[(2)] $D_Y\mu = 0$ for all $Y \in V_\lambda$.
	\item[(3)] The integral curves of $V_\mu$ are geodesics.
	\item[(4)]  If for every unit vector $X\in V_\mu$ there exists a function $f \in C^{\infty}(M)$ such that locally $\nabla f = X$
\end{itemize}
then the metric is a non-trivial warped product with interval base consistent with the eigenspace structure of $A$.
\end{thm}
 
In this theorem, the condition that $\lambda$ be non-constant replaces the condition that $A$ be non-parallel in the constant trace case and guarantees that the warping function is non-trivial. 



\section{Background}\label{sec:background}
	In this section we assemble the tools needed to prove the proposition and the theorem.  All the results in this section can be found in \cite{Derdz} or \cite{Besse} and seem to have appeared first in \cite{Reckziegel}, though the assumptions are slightly different.\footnote{The author wishes to thank Andrzej Derdzinski for bringing this reference to his attention.}  Lemma \ref{lem:same} gives a formula for the image of $\nabla_Y X$ under the Codazzi tensor $A$ given that $X$ and $Y$ are eigenvectors of the same eigenfunction.  

\begin{lem}\label{lem:same}  Suppose $A$ is a Codazzi tensor and that $X$ and $Y$ are two sections in $V_\lambda$.  Then, 
$$A\nabla_Y X = \lambda \nabla_Y X + (D_Y \lambda)X - g(X,Y)\nabla \lambda$$ 
\end{lem}

The next lemma shows that as long as the dimension of the eigendistribution $V_\lambda$ is at least two, the behavior of the eigenfunction $\lambda$ is severely restricted in the sense that its directional derivative is zero along any direction belonging to $V_\lambda$.  When we discuss the particular case of interest where there are only two eigenfunctions $\mu$ and $\lambda$ with dim $V_\mu = 1$ and dim $V_\lambda = n-1$, the lemma implies there's only one linearly independent direction in which $\lambda$ can vary.

\begin{lem}\label{lem:constant}  If $A$ is a Codazzi tensor and $V_\lambda$ is the eigendistribution of the eigenfunction $\lambda$, dim $V_\lambda \ge 2$, then $D_Y\lambda = 0$ for all $Y \in V_\lambda$.
\end{lem}

Lemmas \ref{lem:same} and \ref{lem:constant}  are used to prove the eigendistributions of Codazzi tensors are integrable.
\begin{thm}\label{thm:integrable}  The eigendistributions of a Codazzi tensor $A$ are integrable.

\end{thm}

Note that by combining Theorem \ref{thm:integrable} and Lemma \ref{lem:constant}, one can say that if dim $V_\lambda \ge 2$, then $\lambda$ is constant along the leaves of $V_\lambda$.\\

The final technical lemma gives a formula for the directional derivative of an eigenfunction $\lambda$ when the direction, $Y$, belongs to a different eigendistribution.

\begin{lem}\label{lem:derivative}  If $A$ is a Codazzi tensor with $Y \in V_\lambda$ and $X,Z \in V_\mu$, then
	$$D_Y\mu \cdot g(X,Z) = (\lambda - \mu) g(\nabla_X Y,Z)$$
\end{lem}

\section{Removing the Constant Trace Assumption}\label{sec:mainthm}
We now turn to the proof of Theorem \ref{thm:mainthm}.  The proof consists of two steps.  In Proposition \ref{prop:char}, four conditions are shown to be equivalent.  We then use the proposition to show the existence of a warped product structure.  

It's straightforward to see that the trace of $A$, where $A$ possesses exactly two eigenfunctions $\mu$ and $\lambda$ with dim $V_\mu = 1$ and dim $V_\lambda = n-1$, is $$\textrm{trace }A = \mu + (n-1)\lambda.$$  If the trace is constant and $Y \in V_\lambda$, then by Lemma \ref{lem:constant}, $Y\mu = 0$.  That is, the constant trace assumption implies $Y\mu = 0$ for all $Y \in V_\lambda$.  Theorem \ref{thm:mainthm} shows this conclusion, that $Y\mu = 0$, is sufficient to obtain a warped product structure.
Alternatively, using the first characterization given in Proposition \ref{prop:char}, it's sufficient that the trace be constant in all directions except for one.\\

\begin{prop}\label{prop:char}
	Let $A$ be a Codazzi tensor on a manifold $M^n$, $n \ge 3$, with two distinct eigenfunctions $\mu$ and $\lambda$ on an open domain.  Assume dim $V_\mu =1$.  The following are equivalent.
	\begin{enumerate}
		\item $D_Y \left( tr(A) \right)= 0$ for all $Y \in V_\lambda$.
		\item $D_Y \mu = 0$ for all $Y \in V_\lambda$.
		\item The integral curves of $V_\mu$ are geodesics.
		\item If $X\in V_\mu$, $|X|=1$, then there exists a function $f \in C^{\infty}(M)$ such that locally $\nabla f = X$.
	\end{enumerate}
\end{prop}

\begin{proof}
	Throughout the proof assume $X$ is a unit vector in $V_\mu$ and $Y,Z \in V_\lambda$.\\
	\item[$(1) \Leftrightarrow (2).$]  This follows from the equation $$tr(A) = \mu + (n-1)\lambda$$ and the fact that $D_Y\lambda = 0$ for $Y \in V_\lambda$.\\
\item[$(2) \Leftrightarrow (3)$], $g(\nabla_X X, X) = 0$ since $|X|=1$. By Lemma \ref{lem:derivative}, 
$$D_Y\mu = (\mu - \lambda)g(\nabla_X X, Y)$$
Since $\mu \ne \lambda$, $\nabla_X X = 0$ if and only if $Y\mu = 0$.\\
\item[$(4) \Leftrightarrow (2)$.]  To show this implication, recall that a vector field $X$ is gradient if and only if $\nabla X$ is symmetric. We have,
	\begin{align*}
		(\nabla_Y A)(X,Z) & = - A(\nabla_Y X, Z) - A(X, \nabla_Y Z)\\
					    & = -\lambda g(\nabla_Y X, Z) - \mu g(X, \nabla_Y Z)\\
					    & = (\mu - \lambda)g(\nabla_Y X, Z)\\
	\end{align*}
	
	Similarly, $\displaystyle (\nabla_Z A)(X,Y ) = (\mu - \lambda)g(\nabla_Z X, Y)$.  By the Codazzi condition, $\displaystyle g(\nabla_Z X, Y) = g(\nabla_Y X, Z)$. Thus, as a $(0,2)$ tensor, $\nabla X$ is symmetric on $V_\lambda \times V_\lambda$ regardless of conditions (1) - (3).  We also have,
	
	\begin{align*}
		g(\nabla_X X, Y) & = (\mu -\lambda)^{-1} D_Y\mu\\
		g(\nabla_Y X, X) & = \frac{1}{2}Yg(X,X) = 0
	\end{align*}
This shows $\nabla X$ is symmetric if and only if $Y\mu = 0$.

\end{proof}
We now have everything we need to prove the main theorem.
\begin{proof}[Proof of Theorem \ref{thm:mainthm}]
	In general, the eigenbundles of a Codazzi tensor are integrable and orthogonal, so there exists a chart $\{U,r,y_1,...,y_{n-1}\}$ such that $\partial_r \in V_\mu$ and $\partial_i = \partial_{y_i} \in V_\lambda$.
	  Let $X\in V_{\mu}$ be a unit vector field.  By the fourth criterion in Proposition \ref{prop:char}, $X = \nabla t$ for some local submersion $t:U \to \mathbb{R}$.  Now if $Y\in V_{\lambda}$, then $Yt = g(X,Y) = 0$ meaning there exists coordinates $\{t, x_1,...,x_{n-1}\}$ such that $\partial_t \in V_{\mu}$, $\partial_{x_i} \in V_\lambda$ and $g(\partial_t, \partial_t) = 1$.  Moreover, $A(\partial_t, \partial_j) = g(\partial_t, \partial_j) = 0$ and $A(\partial_i, \partial_j) = \lambda(t)g_{ij}$.
	
	The next step is to prove that $\partial_t g_{ij} = f(t) g_{ij}$.  Lemma \ref{lem:derivative} implies,
	\begin{align*}
		\partial_t g(\partial_i, \partial_j) & = -2g(\nabla_{\partial_i} \partial_j, \partial_t)\\
		& = 2(\mu - \lambda)^{-1}(\partial_t \lambda)g_{ij} = 2 \eta g_{ij} \\
	\end{align*}
	where $\eta = (\mu - \lambda)^{-1}(\partial_t \lambda)$.  We can write,
	$$\eta g_{ij} = -g(\nabla_{\partial_i} \partial_j, \partial_t)  = \textrm{Hess\,t}(\partial_i, \partial_j)$$
	Now show that $\eta$ depends only on $t$.
	\begin{align*}
		\partial_i \eta & = -(\mu-\lambda)^{-2}\cdot(\partial_i \mu -\partial_i \lambda)(\partial_t \lambda) + (\mu - \lambda)^{-1}(\partial_i \partial_t \lambda)\\
		& =  (\mu - \lambda)^{-1}(\partial_i \partial_t \lambda)\\
		& = (\mu - \lambda)^{-1}(\partial_t\partial_i \lambda) = 0\\
	\end{align*}
	
	Since $\partial_t g_{ij} = \eta(t) g_{ij}$, integrate $\eta$ to obtain a function $q(t)$ such that $\partial_t(e^{-q}g_{ij}) = 0$.  This means $g_{ij} = e^{q(t)} h_{ij}$ for some $h_{ij}$.  This shows $M$ is a warped product. The warping function is trivial if and only if $\eta = 0$ which happens if and only if $\lambda$ is constant. 
\end{proof}

\section{Counterexamples}\label{sec:counterexamples}

This section presents a class of Codazzi tensors on open sets of $\mathbb{R}^3$  that provide the source of counterexamples for Propositions \ref{prop:notwarped} and \ref{prop:inconsistent}.

Let $\lambda > 0$ be a constant and $\mu(t,x,y)$ a $C^{\infty}$ function on an open set $V \subset \mathbb{R}^3$ and that there exists a connected open set $U \subset V$ where $\mu \ne \lambda$.  Define a metric and tensor on $U$ by
\begin{align*}
	g & = (\lambda - \mu(t,x,y))^{-2} \, dt^2  + \lambda \, dx^2 +\lambda \, dy^2\\
	A(\partial_t) & = \mu(t,x,y) \partial_t\\
	A(\partial_x) & = \lambda \partial_x\\
	A(\partial_y) & = \lambda \partial_y\\
\end{align*}
As a step toward proving that $A$ is indeed Codazzi, calculate the Christoffel symbols.  Throughout this section, the subscripts $i,j$ and $k$ shall refer to the variables $x$ and $y$.  For example, $\partial_i$ could mean either $\partial_x$ or $\partial_y$ but \emph{not} $\partial_t$.

\begin{lem}\label{lem:christoffels}
The non-trivial Christoffel symbols of this metric are,
\begin{align*}
	\Gamma_{tt}^t & = (\lambda-\mu)^{-1}(\partial_t \mu)\\
	\Gamma_{tt}^i & = -\frac{1}{\lambda}(\lambda - \mu)^{-3}(\partial_i \mu)\\
	\Gamma_{it}^t& = (\lambda - \mu)^{-1}(\partial_i \mu)
\end{align*}
\end{lem}

\begin{prop}\label{prop:codazzi}  The tensor $A$ defined above is Codazzi. \end{prop}

\begin{proof}  It suffices to prove $\displaystyle (\nabla_X A)(Y,Z) = (\nabla_Y A)(X,Z)$ where $X$, $Y$ and $Z$ are all coordinate vectors. Again, let  $\partial_i$, $\partial_j$ and $\partial_k$ indicate partial derivatives with respect to $x$ or $y$.  Straightforward calculations show,
	\begin{itemize}
		\item[i.]$\displaystyle (\nabla_{\partial_t}A)(\partial_i, \partial_j) = (\nabla_{\partial_i}A)(\partial_t, \partial_j) = 0$
	\item[ii.]$\displaystyle (\nabla_{\partial_i}A)(\partial_t, \partial_t) = (\nabla_{\partial_t}A)(\partial_i, \partial_t) = (\partial_i \mu)g_{tt}$
	\item[iii.]$\displaystyle (\nabla_{\partial_i}A)(\partial_j, \partial_k)  = 0$
	\end{itemize}

\end{proof}

Our strategy for proving Propositions \ref{prop:notwarped} and \ref{prop:inconsistent} will be to judiciously select $\mu$ and $\lambda$ along with the following well known characterization of warped products with interval bases, first proved by Brinkmann in \cite{Brinkmann}.

\begin{lem} The following are equivalent
	\begin{itemize}
		\item[A.]  There exists a neighborhood $V$ of $p$ and a function $f$ such that $Hess f =a\cdot g$ for some function $a$ and $\nabla f(p) \ne 0$.
		\item[B.]  There exists a neighborhood $V$ of $p$ such that  $V$ is a warped product space $V = I \times_w F$ with 1-dimensional base $I$.
	\end{itemize}
\end{lem}
Thus, to show a metric is not an interval warped product, it suffices to show if $f$ and $a$ satisfy $Hess f = a \cdot g$, then $\nabla f = 0$ in a neighborhood, i.e. $f$ is locally constant.
	
\begin{lem}\label{lem:hess}  The components of Hessf for the metric given above are,
	\begin{align*}
		Hessf(\partial_x, \partial_x) & = f_{xx}\\
		Hessf(\partial_y, \partial_y) & = f_{yy}\\
		Hessf(\partial_x,\partial_y) & = f_{xy}\\
		Hessf(\partial_t, \partial_t) & = f_{tt}  - (\lambda-\mu)^{-1}\mu_tf_t+\frac{1}{\lambda}\cdot (\lambda-\mu)^{-3}(f_x \mu_x +f_y \mu_y)\\
		Hessf(\partial_t, \partial_x) & = f_{tx} - (\lambda - \mu)^{-1}\mu_x f_t\\
		Hessf(\partial_t, \partial_y) & = f_{ty} - (\lambda- \mu)^{-1}\mu_y f_t\\
	\end{align*}
\end{lem}
\begin{proof}  Use  the formula, $\displaystyle Hessf(\partial_i, \partial_j)  = \partial_i \partial_j f-\Gamma_{ij}^k \partial_k f$ and the Christoffel symbols calculated above.
\end{proof}

From here on will study the particular case where $\lambda = 1$ and $\mu(t,x,y) = \mu(x,y)$.  Then by Lemma 4.2, if we are given a particular $\mu(x,y)$ we should look for functions $f(t,x,y)$ and $a(t,x,y)$ that solve the system below.  
	\begin{align*}
		f_{xx} & = a\\
		f_{yy} & = a\\
		f_{xy} & = 0\\
		f_{ty}(1-\mu) & = \mu_y f_t\\
		f_{tx}(1-\mu) & = \mu_x f_t\\
		f_{tt} + (1 - \mu)^{-3}(f_x \mu_x + f_y \mu_y) & =a (1-\mu)^{-2} \\
	\end{align*}

\begin{prop}
	If for a given $\mu(x,y)$, there exists $f$ satisfying the system of PDE's given above, with $\nabla f(0) \ne 0$, then $\mu$ must be in one of the following forms:
	\begin{itemize}
		\item[1.] $\displaystyle \mu(x,y) = 1 + \frac{c_1}{1-c_3x -c_4y -c_2(x^2+y^2)}$ \\
		 \item[2.] $\displaystyle \mu(x,y) =\frac{ax +G \left(\frac{c+ay}{a(b+ax) } \right)}{ax+b}$  \\
		 \item[3.] $\displaystyle \mu(x,y) =\frac{ay + G\left( \frac{x}{c+ay}\right)}{c+ay}$  \\
		\item[4.] $\displaystyle \mu(x,y) = G\left(\frac{by-cx}{b} \right)$\\
 		 \item[5.] $\displaystyle \mu(x,y) =\mu(x)$\\
		 \item[6.] $\displaystyle \mu(x,y) =\mu(y)$ \\
\end{itemize}
\end{prop}

\begin{proof}
The first three Hessian equations
\begin{align*}
f_{xx} &= ag(\partial_x, \partial_x) = a\\
f_{xy} & = ag(\partial_x, \partial_y) = 0\\
f_{yy} & = a
\end{align*}
collectively imply $a(t,x,y) = a(t)$ and $$f(t,x,y) = \frac{a(t)}{2}\left(x^2+y^2\right)+b(t)x + c(t)y + k(t)$$
If we write $h(t,x,y) = f_t(t,x,y)$, then we can rewrite and solve the fourth Hessian equation as a first order differential equation.
	\begin{align*}
		h_y(1-\mu) & = \mu_y h\\
		h_y & = \mu_y h + \mu h_y = \frac{\partial}{\partial y}(\mu h) \Rightarrow \\
		f_t(t,x,y) & = h(t,x,y) = C_1(t,x)/(\mu(x,y) - 1)\\ 
	\end{align*}
	An analgous argument works for the fifth equation, so 
	$$f_t(t,x,y) = \frac{C_1(t,x)}{\mu(x,y) - 1} = \frac{C_2(t,y)}{\mu(x,y) - 1} \Rightarrow f_t(t,x,y) = \frac{C_1(t)}{\mu(x,y) - 1}$$
We can solve for $C_1(t)$ by using the fact that $f_t(t,x,y)$ is a polynomial and evaluating at $(t,0,0)$.
	\begin{align*}
		\frac{C_1(t)}{\mu(x,y) - 1} & = \frac{a'(t)}{2}(x^2+y^2) + b'(t)x + c'(t) y + k'(t) \Rightarrow \\
		\frac{C_1(t)}{\mu(0,0) - 1} & = k'(t) \Rightarrow \\
		C_1(t) &= (\mu(0,0) - 1)k'(t)\\
	\end{align*}	
	Letting $c_1 = \mu(0,0)-1$, we have
	$$f_t(t,x,y) = \frac{c_1k'(t)}{\mu -1} \Rightarrow f(t,x,y) = \frac{c_1k(t)}{\mu - 1} + K(x,y)$$

Equating to the polynomial expression for $f$ gives us,
	\begin{align*}
		\frac{c_1k(t)}{\mu - 1} + K(x,y) & = \frac{a(t)}{2}(x^2+y^2) + b(t)x + c(t)y + k(t) \Rightarrow \\
		K(x,y) & = \frac{a(t)}{2}(x^2+y^2) + b(t)x + c(t)y + H(x,y)k(t)\\
	\end{align*}
where $\displaystyle H(x,y) =1 - \frac{ c_1}{\mu -1}$.  Take a $t$ derivative of each side to get a linear equation of the functions $x^2+y^2$, $x$, $y$ and $H(x,y)$.
	$$\frac{a'(t)}{2}(x^2+y^2) + b'(t)x + c'(t)y + H(x,y)k'(t) = 0$$
	This equation implies either $a'(t) = b'(t) = c'(t) = k'(t) = 0$ or $$H(x,y) = c_2(x^2+y^2) + c_3 x + c_4 y$$
We can now solve for $\mu$ since,
	\begin{align*}
		H & = 1 - \frac{c_1}{\mu -1} \Rightarrow \\
		\mu(x,y) & = 1 + \frac{c_1}{1-H}\\
			& = 1 + \frac{c_1}{1-c_3x -c_4y -c_2(x^2+y^2)} 	\end{align*}
	A continued analysis would investigate the implications of the final Hessian equation.  However, the equation for $\mu$ just derived will be sufficient for our purposes.\\ 
	
	We now investigate solutions to the system when $a$, $b$, $c$ and $k$ are all constant.  This simplifies the system considerably; the third and fourth equations automatically hold and the final equation simplifies to
	$$(ax+b)\mu_x + (ay +c)\mu_y = a(1-\mu)$$
	
	This PDE is straight forward to solve using the Method of Characteristics.  
	\begin{align*}
		\mu(x,y) & = \frac{ax +G \left(\frac{c+ay}{a(b+ax) } \right)}{ax+b} \textrm{  if $a \ne0$ and $b\ne0$}\\
		\mu(x,y) & = \frac{ay + G\left( \frac{x}{c+ay}\right)}{c+ay}  \textrm{   if $a \ne 0$ and $b = 0$}\\
		\mu(x,y) & = G\left(\frac{by-cx}{b} \right) \textrm{   if $a = 0$ and $b\ne0$}\\
		\mu(x,y) & = \mu(x) \textrm{   if $a = 0$, $b=0$ and $c \ne 0$}\\
		\mu(x,y) & = \mu(y) \textrm{  if $a = 0 $, $b \ne 0$ and $c = 0$}\\
	\end{align*}

\end{proof}

By selecting any $\mu$ \emph{not} in one of the above forms, we can generate an example of a Codazzi tensor on a compact manifold where the metric is not a warped product at at least one point.

\begin{cor}
	Let $\mu(x,y) = \displaystyle \frac{1}{2} \sin x \cos y$.  Then the metric $g$ and Codazzi tensor $A$ defined above are periodic and pass to a metric $\bar{g}$ and Codazzi tensor $\bar{A}$ on $S^1 \times S^1 \times S^1$.  $\bar{g}$ is not a warped product on a neighborhood of the point $\left[(0,0,0)\right]$.
\end{cor}
This result follows from the fact that $\mu(x,y)$ is clearly not in any of the functional forms listed in the proposition.\\

For the proof of Proposition \ref{prop:inconsistent} we use the same template but with $\mu = 1 + (y/x^2)$.  As before $\lambda = 1$.

\begin{proof}[Proof of Proposition \ref{prop:inconsistent}]
Let $M = \{(t,x,y) \in \mathbf{R}^3: y \ne 0\}$.  Define a metric and tensor on $M$ by 
$$\displaystyle g = \frac{x^4}{y^2}dt^2 + dx^2 + dy^2.$$

\begin{align*}
	A \partial_t & = \left(1 + \frac{y}{x^2}\right) \partial_t\\
	A \partial_x & = \partial_x\\
	A \partial_y & = \partial_y
\end{align*}

  It's clear that the metric cannot be written in the form $$g = dt^2 + F(t)\left(dx^2 + dy^2\right)$$ so either the metric is a warped product with inconsistent warping and eigenspace structures, or the metric is not warped at all.    If we let $x = r \cos \theta$ and $y = r \sin \theta$, then 

$$g = dr^2 + r^2\left(\frac{\cos^4 \theta}{\sin^2 \theta}dt^2 + d\theta^2 \right)$$
 This shows the metric is warped in the $r$ direction.
\end{proof}

\end{document}